\documentclass[12pt]{amsart}

\usepackage{amssymb, tikz}

\usepackage{amsmath}

\newtheorem{theorem}{Theorem}[section]
\newtheorem{lemma}[theorem]{Lemma}

\newtheorem{corollary}[theorem]{Corollary}

\theoremstyle{definition}
\newtheorem{definition}[theorem]{Definition}

\newtheorem{remark}[theorem]{Remark}

\let \restr = \upharpoonright
\let \bs = \backslash
\let \into = \longrightarrow

\let \sub = \subseteq

\let \ov = \overline
\let \a = \alpha
\let \b = \beta

\let \d = \delta

\let \k = \kappa

\let \D = \Delta

\let \x = \xi
\let \z= \zeta
\let \o = \omega

\let \la = \langle
\let \ra = \rangle
\let \mtcl = \mathcal

\let \it = \item

\title{Generalized symmetric systems and thin--very tall compact scattered spaces}

\author[M.A. Mota]{Miguel Angel Mota}

\address{Miguel Angel Mota,  Departamento de Matem\'aticas,
Instituto Tecnol\'ogico Aut\'onomo de M\'exico,
Mexico City,
MEXICO
01080}

\email{motagaytan@gmail.com}

\author[W. Weiss]{William Weiss}

\address{William Weiss,  Department of Mathematics,
University of Toronto,
Toronto, Ontario,
CANADA
M5S 2E4}
\email{weiss@math.toronto.edu}

\date{}

\begin{document}

\subjclass[2010]{03E35, 03G05, 54A25, 54A35, 54G12}

\maketitle
\pagestyle{myheadings}\markright{Generalized symmetric systems and thin--very tall sBa}

\begin{abstract}
We solve a well--known problem in the theory of compact scattered spaces and superatomic boolean algebras by showing that, under $\textsc{GCH}$ and for each regular cardinal $\kappa \geq \omega$, there is a poset $\mtcl P_\k$ preserving all cardinals and forcing the existence of a $\kappa$--thin very tall locally compact scattered space. For $\kappa > \o$, we conceive the poset $\mtcl P_\k$ as a higher analogue of the poset $\mtcl P_\o$ originally introduced by Asper\'{o} and Bagaria in the context of an (unpublished) alternative consistency proof.

\end{abstract}

\section{Introduction}

The famous Cantor Bendixson derivative process on a topological space $X$ proceeds by first letting $I_{0}(X)$ be the set of all isolated points of $X$. At each larger ordinal $\beta$, $I_{\beta}(X)$  denotes the set of all isolated points of the subspace $X \setminus \bigcup \{I_{\alpha}(X) : \alpha < \beta \}$. The least ordinal $\lambda$ such that $I_{\lambda}(X) = \emptyset$ is called the height of the space $X$ and the sequence of cardinals $\{ |I_{\beta}(X)| : \beta < \lambda \}$ is said to be the cardinal sequence of $X$.

A natural problem is to determine which sequences of cardinal numbers can be the cardinal sequence of a topological space $X$. The most extensively studied class of such spaces $X$ are the compact Hausdorff scattered spaces; a space $X$ is said to be scattered provided that
$X = \bigcup \{I_{\beta}(X): \beta < \lambda \}$ where $\lambda$ is the height of $X$; that is, its Cantor-Bendixson derivative is the empty set. These spaces are, through the Stone Duality, exactly the Stone spaces of superatomic Boolean algebras and are also related to cardinal arithmetic through the pcf theory of S. Shelah. We defer to \cite{BAG} for a more complete background of the general problem.

The reader will readily observe that the height $\lambda$ of a compact Hausdorff scattered space must be a successor ordinal $\lambda = \alpha + 1$ and that $I_{\alpha}$ must be finite. Hence the problem reduces to determining whether or not a sequence of cardinal numbers is the cardinal sequence of a locally compact Hausdorff scattered space (called an LCS space in the recent literature).

At this time, despite much intensive study, it is still not completely understood which constant sequences of cardinals can be the the cardinal sequences of an LCS space. If the constant sequence with value $\kappa$ and length $\lambda$ is the cardinal sequence of an LCS space $X$ then we say that the space $X$ has width $\kappa$ as well as height $\lambda$. Initially, I. Juh\'asz and the second author (\cite{JUH}) constructed an LCS with width $\omega$ and height $\omega_{1}$, which was dubbed a ``thin--tall" space; under the Continuum Hypothesis (CH) the height cannot be $\omega_{2}$ so the result is, in a sense, the best possible. However, J. Baumgartner and S. Shelah (\cite{BAUM}) used forcing to show that it is consistent (with the negation of the CH) that there is an LCS space of width $\omega$ and height $\omega_{2}$ -- a ``thin--very tall" space.

Generalizing these results from width $\omega$ to width $\kappa$ has proven to be anything but routine. For instance, it is not known whether there is a thin--tall LCS space of width $\omega_{1}$ and height $\omega_{2}$; however such a space is consistent with CH (\cite{KOE}) and also with the negation of CH (\cite{BAGLOCALLY}). Moreover, the consistency of the existence of a thin--very tall LCS space of width $\omega_{1}$ and height $\omega_{3}$ is a well-known problem, stated as Problem 5.3 in \cite{BAG}.

We now give an affirmative answer to Problem 5.3. In fact, we here prove that for each regular cardinal $\kappa$ it is consistent that there is a thin-very tall LCS space of width $\kappa$ and height
$\kappa^{++}$.

Our result also generalizes that of Baumgartner and Shelah ($\kappa = \omega$). However, we borrow from them an important idea.

\begin{definition}\label{admissible}
Let $\unlhd$ be a reflexive, transitive and antisymmetric relation on $\kappa\times \kappa^{++}$. The poset $\langle \kappa \times \kappa^{++}, \unlhd \rangle$ is said to be \emph{$(\kappa,\kappa^{++})$--admissible} iff the following hold:

\begin{itemize}

\it[$(A)$] $(\alpha,\beta) \lhd (\alpha', \beta')$ (i.e., $(\alpha,\beta) \unlhd (\alpha', \beta')$ together with  $(\alpha,\beta) \neq (\alpha', \beta')$) implies $\beta \in \beta'$,
\it[$(B)$] for every $\beta \in  \kappa^{++}$ and every $(\alpha', \beta') \in \kappa \times \kappa^{++}$ with $\beta \in \beta'$, there are infinitely many $\alpha$ such that $(\alpha,\beta) \unlhd (\alpha', \beta')$, and
\it[$(C)$] for all $(\alpha_0,\beta_0)$, $(\alpha_1,\beta_1)$ in $\kappa \times \kappa^{++}$, there is a finite subset $b \subseteq \kappa \times \kappa^{++}$ (called a \emph{barrier for $(\alpha_0,\beta_0)$ and $(\alpha_1,\beta_1)$}) such that

\begin{itemize}
\it[$(C.1)$] all the elements of $b$ are $\unlhd$--below both, $(\alpha_0,\beta_0)$ and $(\alpha_1,\beta_1)$, and

\it[$(C.2)$] every element of $\kappa \times \kappa^{++}$ which is $\unlhd$--below both, $(\alpha_0,\beta_0)$ and $(\alpha_1,\beta_1)$, is $\unlhd$--below some element of $b$.
\end{itemize}

\end{itemize}

\end{definition}

The LCS space $X$ is constructed from the admissible poset. Let $X = \kappa \times \kappa^{++}$. We declare all sets of the form
$$C(\alpha , \beta) = \{ ( \gamma , \delta ) \in \kappa \times \kappa^{++} : ( \gamma , \delta )  \unlhd (\alpha ,
\beta )\}$$

to be clopen, i.e. both open and closed. This generates a base $\mtcl{B}$ for our topology on $X$.

In order to guarantee that $X$ is an LCS space, it suffices to verify the following three claims.
\begin{itemize}
\it $X$ is a Hausdorff space.
\it Each $C(\alpha , \beta)$ is a compact subset of $X$.
\it $I_{\beta}(X) = \kappa \times \{ \beta \}$ for each $\beta < \kappa^{++}$.
\end{itemize}

In order to verify the first claim, let $( \alpha , \beta )$ and $( \alpha' , \beta' )$ be two distinct points. If $( \alpha , \beta ) \unlhd ( \alpha' , \beta' )$ then
$C(\alpha , \beta )$ is a clopen set containing $( \alpha , \beta )$ and not containing $( \alpha' , \beta' )$. Otherwise, $ X \setminus C( \alpha' , \beta' )$ is such a set.

We verify the second claim by induction on $\beta$. Let $\mtcl{C}$ be an open cover of $X$ consisting of clopen sets from our base $\mtcl{B}$. Suppose that

$$ ( \alpha , \beta ) \in B = \bigcap \{ C(\gamma_{i} , \delta_{i}) : i < m \}
\setminus \bigcup \{ C( \eta_{j} , \zeta_{j}) : j < n \}.$$
For each $j \leq n$ obtain a barrier $b_{j}$ for $( \alpha , \beta )$ and
$( \eta_{j} , \zeta_{j} )$. Then

$$ C( \alpha , \beta ) \setminus B \subseteq
\bigcup \{ C(\sigma , \tau) : (\sigma , \tau ) \in \bigcup \{ b_{j} : j \leq n \} \}$$
by property $(C.2)$. By property $(C.1)$ each

$( \sigma , \tau ) \unlhd ( \alpha , \beta )$ and so by property $(A)$ we have
$\tau \in \beta$ which allows us to invoke the inductive hypothesis for each of the finitely many
$C(\sigma , \tau)$.

The third claim is verified by a straightforward induction on $\beta$ using property $(A)$ and property $(B)$.

Our main theorem is the following.

\begin{theorem}\label{mainthm} ($\textsc{GCH}$) For every regular cardinal $\kappa$, there exists a poset $\mtcl P_\k$ such that:
\begin{itemize}
\it[$(a)$] $\mtcl P_\k$ is $<\kappa$--closed and therefore, $\mtcl P_\k$ preserves all cardinals $\leq \kappa$;
\it[$(b)$] $\mtcl P_\k$ is proper with respect to structures of cardinality $\kappa$ and therefore, $\mtcl P_\k$ preserves $\kappa^+$;
\it[$(c)$] $\mtcl P_\k$ has the $\kappa^{++}$--chain condition and therefore, $\mtcl P_\k$ preserves all cardinals $\geq \kappa^{++}$; and
\it[$(d)$] $\mtcl P_\k$ forces the existence of a $(\kappa,\kappa^{++})$--admissible poset.
\end{itemize}
\end{theorem}

\section{Definitions}

From now on, we will say that a set is of \emph{small} size if its cardinality is below $\kappa$. Using this terminology, we say that a forcing poset $\mtcl P$ is \emph{proper with respect to structures of cardinality $\kappa$} iff for every sufficiently large regular cardinal $\lambda$, for every elementary substructure $M \prec H(\lambda)$ of size $\kappa$ containing $P$ and closed under small sequences and for every condition $p \in \mtcl P \cap M$, there is an extension $q \leq_{\mtcl P} p$ such that $q$ is $(M, \mtcl P)$-generic. Note that a forcing poset $\mtcl P$ is proper in the traditional sense iff $\mtcl P$ is proper with respect to $\aleph_0$.

A typical condition in the forcing notion $\mtcl P_\k$ can be described as a small approximation to the $(\kappa,\kappa^{++})$--admissible poset together with a a ``small symmetric system of structures of size $\kappa$'' as side conditions. This notion of symmetry is a generalization of Definition 2.1 in \cite{ASP}.\footnote{Asper\'{o} has also considered the same generalization in \cite{DAV}.} Before stating it, let us introduce some \emph{ad hoc} notation. If $N$ is a set whose intersection with $\kappa^+$ is an ordinal, then $\delta_N$ will denote this intersection. Throughout this paper, if $N$ and $N'$ are such that there is a (unique) isomorphism from $(N, \in)$ into $(N', \in)$, then we denote this isomorphism by $\Psi_{N, N'}$.

\begin{definition}\label{symm}
Let $T \subseteq H(\kappa^{++})$, let $\mu$ be a cardinal and let $\{N_i\,:\,i< \mu\}$ be a small set (so, $\mu < \kappa)$) consisting of subsets of $H(\k^{++})$, each of them of size $\kappa$. We will say that \emph{$\{N_i\,:\,i<\mu\}$ is a $(T,\kappa)$--symmetric system} if

\begin{itemize}

\it[$(A)$] For every $i<\mu$, the set $N_i$ is closed under small sequences and $(N_i, \in, T)$  is an elementary substructure of $(H(\k^{++}), \in, T)$.

\it[$(B)$] Given distinct $i$, $i'$ in $\mu$, if $\d_{N_i}=\d_{N_{i'}}$, then there is a (unique) isomorphism $$\Psi_{N_i, N_{i'}}:(N_i, \in, T)\into (N_{i'}, \in, T)$$

\noindent Furthermore, we ask that $\Psi_{N_i, N_{i'}}$ be the identity on $N_i\cap N_{i'}$.

\it[$(C)$] For all $i$, $j$ in $\mu$, if $\d_{N_j}<\d_{N_i}$, then there is some $i'<\mu$ such that $\d_{N_{i'}}=\d_{N_i}$ and $N_j\in N_{i'}$.

\it[$(D)$] For all $i$, $i'$, $j$ in $\mu$, if $N_j\in N_i$ and $\d_{N_i}=\d_{N_{i'}}$, then there is some $j'<\mu$ such that $\Psi_{N_i, N_{i'}}(N_j)=N_{j'}$.

\end{itemize}

\end{definition}

In (A) in the above definition, and elsewhere, $(N, \in, T)$ denotes the structure  $(N, \in, T\cap N)$. In clause (B) note that, since $\kappa \subseteq N_i\cap N_{i'}$, $\Psi_{N_i, N_{i'}}$ is continuous, in the sense that $sup(\Psi_{N_i, N_{i}}``\,\x)= \Psi_{N_i, N_{i'}}(\x)$ whenever $\x\in N_i$ is a limit ordinal of cofinality at most $\kappa$.

The main facts on symmetric systems that we will be using are the following two amalgamation lemmas. Their proofs are straightforward  verifications and they are word by word the proofs of Lemmas 2.3 and 2.4  in \cite{ASP}, so we will not give them here.

\begin{lemma}\label{iso2}
Let $\mtcl N$ be a $(T,\kappa)$--symmetric system, and let $N\in\mtcl N$. Then the following holds.

\begin{itemize}

\it[(i)] $\mtcl N \cap N$ is also a  $(T,\kappa)$--symmetric system.

\it[(ii)] If $\mtcl W \subseteq  N$ is a $(T,\kappa)$--symmetric system and $\mtcl N \cap N  \subseteq \mtcl W$, then $$\mtcl V:= \mtcl N \cup\{\Psi_{N, N'}(W) \,:\, W \in \mtcl W,\, N' \in \mtcl N,\,\d_{N'}=\d_N\}$$ is a $(T,\kappa)$--symmetric system.
\end{itemize}
\end{lemma}

\begin{lemma}\label{iso3}

Let $T\sub H(\k^{++})$ and let $\mtcl M=\{M_i\,:\,i< \mu\}$ and $\mtcl N=\{N_i\,:\,i<\mu\}$ be $(T,\kappa)$--symmetric systems.
Suppose that $(\bigcup\mtcl M)\cap(\bigcup\mtcl N)=X$ and that there is an isomorphism $\Psi$ between the  structures
$\la \bigcup_{i<\mu}M_i,\in, T, X, M_i\ra_{i<\mu}$ and $\la \bigcup_{i< \mu}N_i,\in,  T, X, N_{i}\ra_{i<\mu}$ fixing $X$. Then  $\mtcl M\cup\mtcl N$ is a $(T,\kappa)$--symmetric system.
\end{lemma}

Fix a bijection $\Phi: \kappa^{++} \into H(\kappa^{++})$. The definition of $\mtcl P_\kappa$ is as follows. Conditions in  $\mtcl P_\kappa$ are sequences of the form $q=(\unlhd_q, b_q, \Delta_q, \Omega_q)$ such that:

\begin{itemize}

\it[$(1)$] $\unlhd_q$ is a small, reflexive, transitive and antisymmetric relation on $dom(\unlhd_q) \cup range(\unlhd_q) \subseteq \kappa\times \kappa^{++}$satisfying condition $(A)$ in definition \ref{admissible},
\it[$(2)$] $b_q$ is a function whose domain is equal to the set $[dom(\unlhd_q)]^2$ such that $b_q(\{(\alpha_0,\beta_0),(\alpha_1,\beta_1)\}$) is a $\unlhd_q$--barrier (in the sense of $(C)$ in definition \ref{admissible}) for $(\alpha_0,\beta_0)$ and $(\alpha_1,\beta_1)$,
\it[$(3)$] $\Delta_q$ is a $(\Phi,\kappa)$--symmetric system,
\it[$(4)$] $\Omega_q \subseteq \Delta_q$, and
\it[$(5)$] if $N$ is in $\Omega_q$ and $\{(\alpha_0,\beta_0), (\alpha_1,\beta_1)\} \subseteq N \cap dom(\unlhd_q)$, then $b_q(\{(\alpha_0,\beta_0), (\alpha_1,\beta_1)\}) \in N$.

\end{itemize}

From now on, if $q$ is a sequence with four ordered elements, then $\unlhd_q, b_q, \Delta_q$ and $\Omega_q$ will respectively denote the first, second, third and fourth element of $q$. Given conditions $q$ and $p$ in $\mtcl P_\kappa$, $q$ extends $p$  (denoted by $q \leq_\kappa p$) if and only if $\unlhd_p, b_p, \Delta_p$ and $\Omega_p$ are respectively included in $\unlhd_q, b_q, \Delta_q$ and $\Omega_q$.

In an unpublished note dated September 2009 David Asper\'{o} and Joan Bagaria introduced the forcing
$\mtcl P_\kappa$ as above for $\kappa= \omega$. We would like to thank them for sharing this specific construction with us.

However, the constraint imposed by the finiteness of the barrier function is not at all problematic in the context $\kappa= \omega$ since the approximations to the $(\omega,\omega_2)$--admissible poset coming from conditions in the forcing $\mtcl P_\omega$ are also finite. However, it is not clear how to amalgamate two infinite partial orders in such a way that it is still possible to define a barrier function. For this reason, we needed to introduce the notion of \emph{progressively isomorphic partial orders} (see Definition \ref{progressive}) as well as Lemmas \ref{amalgorder} and \ref{amalgcond}. These results are the main ingredients allowing us to prove that, for every regular cardinal $\kappa$, $\mtcl P_\kappa$ is $\kappa^{++}$--c.c. and proper with respect to structures of cardinality $\kappa$.

\section{Proving Theorem \ref{mainthm}}

The following two lemmas are immediate (for the first one it suffices to recall that $\kappa$ is a regular cardinal)

\begin{lemma}
$\mtcl P_\kappa$ is $<\kappa$--closed.
\end{lemma}

\begin{lemma}
If $q$ is a condition in $\mtcl P_\kappa$, the pairs $(\alpha, \beta)$ and $(\alpha', \beta')$ are in  $\kappa \times \kappa^{++}$, $(\alpha, \beta)$ is not in the domain of $\unlhd_q$ and $\beta \in \beta'$, then there exists $q' \leq_\kappa q$ such that $(\alpha, \beta) \unlhd_{q'}(\alpha', \beta')$.
\end{lemma}

\begin{lemma}\label{amalgorder}
Let $\langle P_1, \unlhd_1 \rangle$ and $\langle P_2, \unlhd_2 \rangle$ be posets and let $\Psi:\langle P_1, \unlhd_1 \rangle \into \langle P_2, \unlhd_2 \rangle$ be an isomorphism fixing $P_1 \cap P_2$. Then there is a (unique) partial order $\unlhd_3$ on $P_1 \cup P_2$ (called the \emph{$\Psi_{1,2}$--amalgamation of $\unlhd_1$ and $\unlhd_2$}) satisfying:
\begin{itemize}
\it[(i)] for all $x$, $y \in P_1$, $x \unlhd_3 y$ iff $x \unlhd_1 y$,
\it[(ii)] for all $x$, $y \in P_2$, $x \unlhd_3 y$ iff $x \unlhd_2 y$,
\it[(iii)] for all $x \in P_1 \setminus P_2$ and $y \in P_2 \setminus P_1$, $x \unlhd_3 y$ iff $\Psi(x) \unlhd_2 y$, and
\it[(iv)] for all $x \in P_2 \setminus P_1$ and $y \in P_1 \setminus P_2$, $x \unlhd_3 y$ iff $x \unlhd_2 w \unlhd_1 y$ for some $w \in P_1 \cap P_2$.
\end{itemize}

\end{lemma}
\begin{proof}
For $x$, $y \in P_1 \cup P_2$, let $x \unlhd_3 y$ provided at least one of the above four conditions holds. Note that the first two conditions are compatible since $\Psi$ fixes $P_1 \cap P_2$. Let us verify that this order is a partial order on $P_1 \cup P_2$. Trivially, $\unlhd_3$ is reflexive on $P_1 \cup P_2$. In order to verify antisymmetry, it suffices to show that there can be no
$x \in P_1 \setminus P_2$ and $y \in P_2 \setminus P_1$ with $x \unlhd_3 y$ and $y \unlhd_3 x$. If there were such $x$ and $y$, then  we would have $\Psi(x) \unlhd_2 y$ and $y \unlhd_2 w \unlhd_1 x$ for some $w \in P_1 \cap P_2$. So, $\Psi(x) \unlhd_2 y \unlhd_2 w =\Psi(w) \unlhd_2 \Psi(x) \unlhd_2 y$. Using the antisymmetry of $\unlhd_2$ we would conclude that $y =w$ contradicting the assumption that $y \notin P_1$. We finally check the transitivity of $\unlhd_3$. So, assume that $x \unlhd_3 y$ and $y \unlhd_3 z$; we will see that $x \unlhd_3 z$. The proof breaks into nine cases, depending upon where $x$ and $y$ lie.

Case 1: $x$, $y \in P_1 \cap P_2$. This first case is straightforward using the transitivity of $\unlhd_1$ (or the transitivity of  $\unlhd_2$) depending whether $z \in P_1$ (or $z \in P_2$).

Case 2: $x \in P_1 \cap P_2$ and $y \in P_1 \setminus P_2 $ (and hence, $x \unlhd_1 y$). We omit the obvious subcase when $z \in P_1$. If $z \in P_2 \setminus P_1$, then $x = \Psi(x) \unlhd_2 \Psi(y) \unlhd_2 z$. So, $x \unlhd_2 z$ and $x \unlhd_3 z$.

Case 3: $x \in P_1 \cap P_2$ and $y \in P_2 \setminus P_1$  (and hence, $x \unlhd_2 y$). We omit the obvious subcase when $z \in P_2$. If $z \in P_1 \setminus P_2$, then $x \unlhd_2 y \unlhd_2 w \unlhd_1 z$ for some $w \in P_1 \cap P_2$. So, $x \unlhd_2 w$, $x \unlhd_1 w \unlhd_1 z$ and $x \unlhd_3 z$.

Case 4: $x \in P_1 \setminus P_2$ and $y  \in P_1 \cap P_2$ (and hence, $x \unlhd_1 y$). We omit the obvious subcase when $z \in P_1$. If $z \in P_2 \setminus P_1$, then $\Psi(x) \unlhd_2 \Psi(y)= y \unlhd_2 z$. So, $x \unlhd_3 z$.

Case 5: $x \in P_1 \setminus P_2$ and $y  \in P_1 \setminus P_2$ (and hence, $x \unlhd_1 y$). We omit the obvious subcase when $z \in P_1$. If $z \in P_2 \setminus P_1$, then $\Psi(x) \unlhd_2 \Psi(y) \unlhd_2 z$. So, $x \unlhd_3 z$.

Case 6: $x \in P_1 \setminus P_2$ and $y  \in P_2 \setminus P_1$ (and hence, $\Psi(x)  \unlhd_2 y$). If $z \in P_1 \cap P_2$, then $\Psi(x) \unlhd_2 y \unlhd_2 z = \Psi(z)$. So, $x \unlhd_1 z$ and $x \unlhd_3 z$. If $z \in P_1 \setminus P_2$, then $y \unlhd_2 w \unlhd_1 z$ for some $w \in P_1 \cap P_2$. In this subcase we get $\Psi(x) \unlhd_2y \unlhd_2 w= \Psi(w)$. So, $x \unlhd_1 w \unlhd_1 z$ which implies that $x \unlhd_1 z$ and $x \unlhd_3 z$. Finally, if $z \in P_2 \setminus P_1$, then $\Psi(x) \unlhd_2 y \unlhd_2 z$. So, $\Psi(x) \unlhd_2 z$ and $x \unlhd_3 z$.

Case 7:  $x \in P_2 \setminus P_1$ and $y \in P_1 \cap P_2$ (and hence, $x \unlhd_2 y$). We omit the obvious subcase when $z \in P_2$. If $z \in P_1 \setminus P_2$, then $x \unlhd_2 y \unlhd_1 z$. So, $x \unlhd_3 z$.

Case 8: $x \in P_2 \setminus P_1$ and $y \in P_1 \setminus P_2$ (and hence, $x \unlhd_2 w \unlhd_1 y$ for some $w \in P_1 \cap P_2$). If $z \in P_1$, then $w \unlhd_1 z$. So, if $z \in P_1 \setminus P_2$, then $x \unlhd_2 w \unlhd_1 z$ and $x \unlhd_3 z$. If $z \in P_1 \cap P_2$, then $x \unlhd_2 w = \Psi(w) \unlhd_2 \Psi(z)=z$ and $x \unlhd_3 z$. Finally, if $z \in P_2 \setminus P_1$, then $x \unlhd_2 w = \Psi(w) \unlhd_2 \Psi(y)  \unlhd_2 z$. So, $x \unlhd_2 z$ and $x \unlhd_3 z$.

Case 9: $x \in P_2 \setminus P_1$ and $y \in P_2 \setminus P_1$ (and hence, $x \unlhd_2 y$). We omit the obvious subcase when $z \in P_2$. If $z \in P_1 \setminus P_2$, then $x \unlhd_2 y \unlhd_2 w  \unlhd_1 z $ for some $w \in P_1 \cap P_2$. Obviously, $x \unlhd_3 z$ which completes the proof of the lemma.

\end{proof}

\begin{definition}\label{progressive}
 Let $p_1=(\unlhd_1, b_1, \emptyset, \emptyset)$ and $p_2=(\unlhd_2, b_2, \emptyset, \emptyset)$ be $\mtcl P_\kappa$--conditions and  let $\Psi_{1,2}: \langle dom(\unlhd_1), \unlhd_1, b_1\rangle \into \langle dom(\unlhd_2), \unlhd_2, b_2 \rangle $ be an isomorphism.  $\Psi_{1,2}$ is said to be a \emph{progressive isomorphism} iff the following three conditions hold:
\begin{itemize}

\it[$(\oplus)$] if $(\a,\b)$ is in $dom(\unlhd_1)$, then there is $\beta' \geq \beta$ such that $\Psi_{1,2}((\a,\b))= (\a,\b')$,

\it[$(\ominus)$] if $(\a,\b)$ is in $dom(\unlhd_1)$ and $\beta$ is in the range of $dom(\unlhd_2)$, then $\Psi_{1,2}((\a,\b))= (\a,\b)$, and

\it[$(\otimes)$] if $u$ and $v$ are both in $dom(\unlhd_1)\cap dom(\unlhd_2)$, then $b_1(\{u,v\})= b_2(\{u,v\})$.

\end{itemize}
\end{definition}

The following result is obvious.

\begin{remark}
Let $p_1=(\unlhd_1, b_1, \emptyset, \emptyset)$ and $p_2=(\unlhd_2, b_2, \emptyset, \emptyset)$ be $\mtcl P_\kappa$--conditions. If $\Psi_{1,2}$ is a progressive isomorphism from  $\langle dom(\unlhd_1), \unlhd_1, b_1\rangle$ to $\langle dom(\unlhd_2), \unlhd_2, b_2 \rangle $, then $\Psi_{1,2}$ fixes $dom(\unlhd_1) \cap dom(\unlhd_2)$ and $(\unlhd_3, \emptyset, \emptyset, \emptyset)$ is also in $\mtcl P_\kappa$, where $\unlhd_3$ is the $\Psi_{1,2}$--amalgamation of $\unlhd_1$ and $\unlhd_2$ as in Lemma \ref{amalgorder}.
\end{remark}

\begin{lemma}\label{amalgcond}
Let $p_1=(\unlhd_1, b_1, \emptyset, \emptyset)$ and $p_2=(\unlhd_2, b_2, \emptyset, \emptyset)$ be $\mtcl P_\kappa$--conditions, let $\Psi_{1,2}: \langle dom(\unlhd_1), \unlhd_1, b_1\rangle \into \langle dom(\unlhd_2), \unlhd_2, b_2 \rangle $ be a progressive isomorphism and let $\unlhd_3$ be the $\Psi_{1,2}$--amalgamation of $\unlhd_1$ and $\unlhd_2$. Using recursion on the ordinals define the (unique) function $b_3: [dom(\unlhd_1) \cup dom(\unlhd_2)]^2 \into [dom(\unlhd_1) \cup dom(\unlhd_2)]^{\omega}$ satisfying the following properties:

\begin{itemize}

\it[$(i)$] $b_3$ extends both,  $b_1$ and $b_2$,

\it[$(ii)$] if $x \in dom(\unlhd_1)\setminus dom(\unlhd_2)$, then $b_3(\{x, \Psi_{1,2}(x)\})=\{x\}$, and

\it[$(iii)$] if $x \in dom(\unlhd_1)\setminus dom(\unlhd_2)$, $y \in dom(\unlhd_2)\setminus dom(\unlhd_1)$ and $y \neq \Psi_{1,2}(x)$, then $b_3(\{x, y\})=\bigcup \{b_3(\{x, v\})\,:\,v \in b_2(\{\Psi_{1,2}(x), y\})\} \cup \bigcup \{b_3(\{u, y\})\,:\,u \in b_1(\{x, \Psi_{1,2}^{-1}(y)\})\}$.

\end{itemize}
Then $p_3:=(\unlhd_3, b_3, \emptyset, \emptyset)$ is a condition in $\mtcl P_\kappa$, extending both $p_1$ and $p_2$. This condition will be called the $\Psi_{1,2}$--amalgamation of $p_1$ and $p_2$.
\end{lemma}

\begin{proof}
By the above remark, it is enough to check that $p_3$ satisfies condition $(2)$ of the definition of the elements of $\mtcl P_\kappa$. For so doing, note first that an inductive argument (together with the fact the range of $b_i$ ($i \in \{1,2\}$ is included in $[dom(\unlhd_i)]^{< \omega}$) allows to show that the range of $b_3$ is certainly included in $[dom(\unlhd_1) \cup dom(\unlhd_2)]^{\omega}$. We check now that every element of $b_3(\{x, y\})$ is $\unlhd_3$ below both, $x$ and $y$ which is obviously true when either there is $i$ in $\{1,2\}$ such that $x$ and $y$ are both in $dom(\unlhd_i)$, or  $x \in dom(\unlhd_1) \setminus dom(\unlhd_2)$ and $y=\Psi_{1,2}(x)$ (in this last case $x \unlhd_3 y= \Psi_{1,2}(x)$ since $\Psi_{1,2}(x) \leq_2 \Psi_{1,2}(x)$). The remaining case when $x \in dom(\unlhd_1)\setminus dom(\unlhd_2)$ and $y \in dom(\unlhd_2)\setminus dom(\unlhd_1)$ can be easily proved by induction using the transitivity of $\unlhd_3$. Now we check that every element $t$ of $dom(\unlhd_1) \cup dom(\unlhd_2)$ which is $\unlhd_3$--below  both, $x$ and $y$, is $\unlhd_3$--below  some element of $b_3(\{x,y\})$. We will inductively verify four possible cases.

Case 1: $x \in dom(\unlhd_1)$ and $y \in dom(\unlhd_1)$. We omit the obvious subcase when $t \in dom(\unlhd_1)$. If $t \in dom(\unlhd_2) \setminus dom(\unlhd_1)$, then there are $u$ and $v$ in $dom(\unlhd_1) \cap dom(\unlhd_2)$ with $t \unlhd_2 u \unlhd_1 x$ and $t \unlhd_2 v \unlhd_1 y$. If $u=v$, we let $w=u=v$. If $u \neq v$, we find $w \in b_2(\{u,v\})$ such that $t \unlhd_2 w$. Given that $b_1(\{u,v\})=b_2(\{u,v\})$, $w$ is in $dom(\unlhd_1) \cap dom(\unlhd_2)$. So, it is always possible to find $z \in b_1(\{x,y\})$ with $w \unlhd_1 z$. Since $t \unlhd_2 w \unlhd_1 z$, we conclude $t \unlhd_3 z$.

Case 2:  $x \in dom(\unlhd_2)$ and $y \in dom(\unlhd_2)$. We omit the obvious subcases when either $t \in dom(\unlhd_2)$, or $x, y \in dom(\unlhd_2) \cap dom(\unlhd_1)$. So, assume first that $x,y \in dom(\unlhd_2) \setminus dom(\unlhd_1)$ and $t \in dom(\unlhd_1) \setminus dom(\unlhd_2)$.  Under these assumptions, $\Psi_{1,2}(t) \unlhd_2 x, y$ and we can find $z \in b_2(\{x,y\})=b_3(\{x,y\})$ such that $\Psi_{1,2}(t)\unlhd_2 z$ and hence, $t \unlhd_3 z$. The final subcase to consider is when $x \in dom(\unlhd_2) \setminus dom(\unlhd_1)$, $y \in dom(\unlhd_2) \cap dom(\unlhd_1)$ and $t \in dom(\unlhd_1) \setminus dom(\unlhd_2)$. In this scenario, $t \unlhd_1 y= \Psi_{1,2}(y)$ and hence $\Psi_{1,2}(t) \unlhd_2 x, y$. The rest of the argument is just as before.

Case 3: $x \in dom(\unlhd_1) \setminus dom(\unlhd_2)$ and $y=\Psi_{1,2}(x)$. This case is obvious since $b_3(\{x,y\})=\{x\}$.

Case 4: $x \in dom(\unlhd_1) \setminus dom(\unlhd_2)$, $y \in dom(\unlhd_2) \setminus dom(\unlhd_1)$ and $y \neq \Psi_{1,2}(x)$. If $t \in dom(\unlhd_1) \setminus dom(\unlhd_2)$, then $t \unlhd_1 x$ and $\Psi_{1,2}(t)\unlhd_2 y$.  In particular, $t \unlhd_1 \Psi_{1,2}^{-1}(y)$ and there is $u \in  b_1(\{x, \Psi_{1,2}^{-1}(y)\})$ with $t \unlhd_1 u$. Note that $u   \in dom(\unlhd_1)$ and hence, $u \neq y$. Using either Case 2 or the inductive hypothesis (depending whether or not $u \in dom(\unlhd_2)$ we can always find $z \in b_3(\{u,y\}) \subseteq b_3(\{x,y\})$ with $t \unlhd_3 z$. If $t \in dom(\unlhd_2)$, then $t \unlhd_2 y$ and $t \unlhd_2 w \unlhd_1 x$ for some $w  \in dom(\unlhd_1) \cap dom(\unlhd_2)$ ($w=t$ when $t \in dom(\unlhd_2) \cap dom(\unlhd_1)$). In particular, $t \unlhd_2 w= \Psi_{1,2}(w) \unlhd_2 \Psi_{1,2}(x)$ and there is $v \in  b_2(\{\Psi_{1,2}(x), y\})$ with $t \unlhd_2 v$. Note that $v   \in dom(\unlhd_2)$ and hence, $v \neq x$. Using either Case 1 or the inductive hypothesis (depending whether or not $v \in dom(\unlhd_1)$ we can always find $z \in b_3(\{x,v\}) \subseteq b_3(\{x,y\})$ with $t \unlhd_3 z$. This ends the proof of the lemma.
\end{proof}

Assume again the hypothesis of Lemma \ref{amalgcond}. In the following we will define a useful variation of the barrier $b_3$ defined above. First, adopt the natural convention that $b_{1}(\{x , x \}) = \{ x \}$ whenever $x \in dom(\unlhd_1)$. So, for such an $x$, $b_3(\{x, \Psi_{1,2}(x)\})=\{x\}= b_{1}(\{x , x \})$. A straightforward induction shows that if $x \in dom(\unlhd_1)\setminus dom(\unlhd_2)$, $y \in dom(\unlhd_2)\setminus dom(\unlhd_1)$ and $y \neq \Psi_{1,2}(x)$, then there are two (canonical) finite sequences $\langle\{x_i, v_i\} \,:\, i < m \rangle$ and $\langle\{u_j, y_j\} \,:\, j < n \rangle$ with  $x_i, u_j \unlhd_1 x$ and
$v_i, y_j \unlhd_2 y$ for all $i$ and $j$ and such that

$$b_3(\{x, y\})=\bigcup \{b_1(\{x_i, v_i\})\,:\, i <m \} \cup \bigcup \{b_2(\{u_j, y_j\})\,:\, j < n \}.$$

In particular, $v_i$ and $u_j$ are in $dom(\unlhd_1)\cap dom(\unlhd_2)$ for all $i$ and $j$. Now, note that if $j <n$, then each element of $b_2(\{u_j, y_j\})$ is $\unlhd_2$--below both, $u_j$ and $y$ and therefore, $\unlhd_2$--below some element of $b_2(\{u_j, y\})$. So, if we define $B_3(\{x, y\})$ as

$$B_3(\{x, y\})=\bigcup \{b_1(\{x_i, v_i\})\,:\, i <m \} \cup \bigcup \{b_2(\{u_j, y\})\,:\, j < n \}$$

\noindent
then the conclusion is that every element $t$ of $dom(\unlhd_1) \cup dom(\unlhd_2)$ which is $\unlhd_3$--below both, $x$ and $y$, is $\unlhd_3$--below some element of $B_3(\{x,y\})$ (and of course every element of $B_3(\{x, y\}$ is $\unlhd_3$--below both, $x$ and $y$). More generally, we have proved the following result.

\begin{corollary}\label{amalgcond2}
Let $p_1=(\unlhd_1, b_1, \emptyset, \emptyset)$ and $p_2=(\unlhd_2, b_2, \emptyset, \emptyset)$ be $\mtcl P_\kappa$--conditions, let $\Psi_{1,2}: \langle dom(\unlhd_1), \unlhd_1, b_1\rangle \into \langle dom(\unlhd_2), \unlhd_2, b_2 \rangle $ be a progressive isomorphism and let $\unlhd_3$ be the $\Psi_{1,2}$--amalgamation of $\unlhd_1$ and $\unlhd_2$. Using recursion on the ordinals define the (unique) function $B_3: [dom(\unlhd_1) \cup dom(\unlhd_2)]^2 \into [dom(\unlhd_1) \cup dom(\unlhd_2)]^{\omega}$ satisfying the following properties:

\begin{itemize}

\it[$(i)$] $B_3$ extends both,  $b_1$ and $b_2$,

\it[$(ii)$] if $x \in dom(\unlhd_1)\setminus dom(\unlhd_2)$, then $B_3(\{x, \Psi_{1,2}(x)\})=\{x\}$, and

\it[$(iii)$] if $x \in dom(\unlhd_1)\setminus dom(\unlhd_2)$, $y \in dom(\unlhd_2)\setminus dom(\unlhd_1)$ and $y \neq \Psi_{1,2}(x)$, then $B_3(\{x,y\})$ is defined as in the previous discussion.

\end{itemize}
Then $p_3:=(\unlhd_3, B_3, \emptyset, \emptyset)$ is a condition in $\mtcl P_\kappa$, extending both $p_1$ and $p_2$. This condition will be called the $\Psi_{1,2}$--Amalgamation of $p_1$ and $p_2$. Note the upper case A.
\end{corollary}

\begin{lemma}\label{c.c.}

$\mtcl P_\kappa$ has the $\kappa^{++}$--chain condition.
\end{lemma}

\begin{proof}
Suppose that $q_\x=(\unlhd_\x, b_\x, \Delta_\x, \Omega_\x)$ is a $\mtcl P_\kappa$--condition for each $\x<\kappa^{++}$. We should find two $\leq_\kappa$--compatible conditions. For so doing we will use standard $\D$--system and pigeonhole principle arguments. We may assume that for each $\x \in \kappa^{++}$, there exists $\overline{N}_\x\in \Delta_\x$ such that $\{\unlhd_\x, b_\x, \Delta_\x \setminus\{\overline{N}_\x\} \} \subseteq \overline{N}_\x$. A second reduction consists in assuming that there are cardinals $\lambda$ and $\mu$ below $\kappa$ such that for all $\x<\kappa^{++}$, $|\unlhd_{\x}|= \lambda$ and $|\D_\x|=\mu$.
Suppose $\Delta_\x=\{N_{\x,i}\,:\,i< \mu\}$.  By $\textsc{GCH}$ we may assume that $\{\bigcup \Delta_\x \,:\,\x<\kappa^{++}\}$ forms a $\Delta$--system with root  $R$ (note that for each $\x$, $\bigcup \Delta_\x =\overline{N}_\x$). Furthermore, by $\textsc{GCH}$ we may assume, for all $\x$, $\x'<\kappa^{++}$, that the structures
$$W_\x:=\la \overline{N}_\x, \in, \Phi, R, (N_{\x,i})_{i \in \mu}, (N)_{N\in \Omega_\x}, \unlhd_\x, b_\x\ra$$
are pairwise isomorphic and that the corresponding isomorphism fixes $R$. The first assertion follows from the fact that there are only $\kappa^+$--many isomorphism types for such structures. For the second assertion note that, if $\Psi$ is  the unique isomorphism between $W_\x$ and $W_{\x'}$, then the restriction of $\Psi$ to $R\cap\k^{++}$ has to be the identity on $R\cap\k^{++}$. Since  there is  a bijection $\Phi: \k^{++} \into H(\k^{++})$ definable in $(H(\k^{++}), \in, \Phi)$, we have that $\Psi$ fixes $R$ if and only if it fixes $R\cap\k^{++}$. It follows that $\Psi$ fixes $R$.
Hence, by Lemma \ref{iso3} we have, for all $\x$, $\x'$ in $\kappa^{++}$, that $\D_\x\cup \D_{\x'}$ is a $(\Phi,\kappa)$--symmetric system.\footnote{A similar argument appears in the proof of Lemma 3.9 in \cite{ASP}.}

Let $\sigma_0$ be equal to the supremum of $\Phi^{-1}[R]$ (since $\Phi$ is a bijection between $\kappa^{++}$ and $H(\kappa^{++})$, it is clear that $\Phi^{-1}[R]$= $R \cap \kappa^{++}$). By $\textsc{GCH}$, we may assume that for all $\x$ and $\x'$ in $\kappa^{++}$, $\unlhd_\x$ and $b_\x$ restricted to $\Phi[\sigma_0] \times \Phi[\sigma_0] $ are respectively equal to $\unlhd_{\x'}$ and $b_{\x'}$  restricted to $\Phi[\sigma_0] \times \Phi[\sigma_0] $. Find $\x_1 <\x_2$ in $\kappa^{++}$ such that $\sigma_0 \cap \bigcup \Delta_{\x_1} = \sigma_0 \cap R$ and such that, letting $\sigma_1$ be equal to the supremum of $\kappa^{++} \cap \bigcup \D_{\x_1}$, $\sigma_1 \cap \bigcup \Delta_{\x_2} = \sigma_0 \cap R$. As a convenient abuse of notation, $q_i=(\unlhd_i, b_i, \Delta_i, \Omega_i)$ ($i \in\{1,2\}$) will denote $q_{\x_i}=(\unlhd_{\x_i}, b_{\x_i}, \Delta_{\x_i}, \Omega_{\x_i})$. Accordingly, $\Psi_{1,2}$ will denote the isomorphism between the structures $W_{\x_1}$ and $W_{\x_2}$.

We now claim that $q_{1}$ and $q_{_2}$ are compatible. We have already check that $(\emptyset, \emptyset, \Delta_1 \cup \Delta_2, \emptyset)$ is a condition in $\mtcl P_\kappa$. Also note that, by the choice of $\sigma_0$ and $\sigma_1$ together with the fact that $\kappa$ is included in $\overline{N}_1 \cap \overline{N}_2$,  $\Psi_{1,2}$ is a progressive isomorphism between $p_1=(\unlhd_1, b_1, \emptyset, \emptyset)$ and $p_2=(\unlhd_2, b_2, \emptyset, \emptyset)$. So, if $p_3:=(\unlhd_3, b_3, \emptyset, \emptyset)$ is defined as $\Psi_{1,2}$--amalgamation of $p_1$ and $p_2$ (as in Lemma \ref{amalgcond}), then $p_3 \in \mtcl P_\kappa$ and therefore, $(\unlhd_3, b_3, \Delta_1 \cup \Delta_2, \emptyset) \in \mtcl P_\kappa$. This proof will be completed once we  show that $q_3=(\unlhd_3, b_3, \Delta_1 \cup \Delta_2, \Omega_1 \cup \Omega_2)$ satisfies condition $(5)$ of the definition of the elements of $\mtcl P_\kappa$.

So, fix $N \in \Omega_1 \cup \Omega_2$ and $u, v  \in (dom(\unlhd_1) \cup dom(\unlhd_2))\cap N$. We must show that $b_3(\{u,v\})$ is an element of $N$. We omit the obvious case when there exists $j \in \{1,2\}$ such that $N \in \Delta_{j}$ and $u,v \in dom(\unlhd_j)$. So, assume first that $i$ and $j$ are two different elements of the set $\{1, 2\}$, $ N \in \Delta_i$ and $u,v \in dom(\unlhd_{j})$. So, $u$ and $v$ are members of the root $R$ and $\{u,v\} \subseteq   \Psi_{i,j}(N) \in \Omega_j$ (recall that $\Psi_{1,2}$ fixes $R$). Moreover, by the choice of $\sigma_0$ and $\sigma_1$, it is also true that $b_1(\{u,v\})=b_2(\{u,v\})$. Also note that $q_j$ satisfies condition $(5)$ in the definition of our forcing and hence, $b_j(\{u,v\}) \in \Psi_{i,j}(N)$. Using this relation and the fact that the isomorphism $\Psi_{j,i}$ fixes $R$, we deduce that $b_3(\{u,v\})=b_i(\{u,v\}) \in N$. The last case to consider is when $N \in \Omega_i$, $u \in dom(\unlhd_{i}) $ and $v  \in dom(\unlhd_{j})$. But this case is straightforward, since $v$ is in the root $R$ and therefore, $v$ is also a member of $dom(\unlhd_{i})$.
\end{proof}

The following lemma will be used to argue that $\mtcl P_\k$ is proper with respect to structures of cardinality $\kappa$ and it was taken from a primitive version of \cite{ASP}.

\begin{lemma}\label{fff}
Let $\mtcl N =\{N_i\,:\,i<\mu\}$ be a $(T,\kappa)$--symmetric system. For every $N\in\mtcl N$ and every $i\in \k^{++}\bs N$ there are ordinals $\a_i<\b_i$ such that

\begin{itemize}
\it[(a)] $\a_i\in N$,

\it[(b)] $\b_i= min(N\bs i)$ if this minimum exists, and $\b_i= \k^{++}$ otherwise,

\it[(c)] $\a_i<i<\b_i$, and

\it[(d)] $[\a_i,\,\b_i)\cap N'\cap N=\emptyset$ whenever $N'\in\mtcl N \bs N$ is such that $\d_{N'}<\d_N$.

\end{itemize}

\end{lemma}

\begin{proof}
Since the set of $N'\in\mtcl N \bs N$ such that $\d_{N'}<\d_N$ is small and $N$ is closed under small sequences, it suffices to show, for a fixed such $N'$, that there is an ordinal $\a_i$ in  $i \cap N$ satisfying $[\a_i,\,\b_i)\cap N'\cap N=\emptyset$.

Let $\ov N \in\mtcl N$ be such that $N'\in \ov N$ and $\d_{\ov N}=\d_N$. Suppose first that $\eta=min((\ov N\cap \k^{++})\bs sup(N\cap \ov N\cap i))$ exists. In that case, $cf(\eta)=\kappa^+$ (if $cf(\eta) \leq \kappa$, then $\eta$ is a limit point of ordinals in $N\cap \ov N$, and therefore $\Psi_{\ov N, N}(\eta)=\eta\in N$ since $\Psi_{\ov N, N}$ is continuous and fixes all ordinals in $\ov N\cap N$. So, $\eta$ and its successor are members of $N \cap \ov N \cap i$ which is of course a contradiction). But then $\z=sup(N'\cap \eta)<\eta$. Hence, since $\z\in \ov N$, there is some $\a_i\in N\cap \ov N\cap i$ above $\z$. Then we have that $\a_i<i<\b_i$ by the choice of these ordinals (and the fact that $i\notin N$). Also, if $i^\ast\in [\a_i,\,\b_i)\cap N'\cap N$, then $i^\ast\in \ov N\cap N$ and $\eta\leq i^\ast$. Hence $i<i^\ast$, and therefore $\b_i\leq i^\ast$, which is a contradiction.

The other possibility is that $N\cap \ov N\cap i$ is cofinal in $\ov N\cap \k^{++}$. Again, since $sup(N'\cap i) =sup(N'\cap \k^{++})\in\ov N$, we may fix $\a_i\in N\cap \ov N\cap i$ above $sup(N'\cap i)$. The rest of the proof is now as in the previous case.
\end{proof}

The following lemma ends the proof of Theorem \ref{mainthm}.

\begin{lemma}\label{proper}
$\mtcl P_\kappa$ is proper with respect to structures of cardinality $\kappa$.
\end{lemma}

\begin{proof}
Let $\lambda$ be a sufficiently large regular cardinal $\lambda$, let $N^* \prec H(\lambda)$ be an elementary substructure of size $\kappa$ containing $\mtcl P_\kappa$ and closed under small sequences and let $p \in \mtcl P \cap N^*$. We should find an extension $q \leq_{\mtcl P_\kappa} p$ such that $q$ is $(N^*, \mtcl P_\kappa)$-generic. Let $q$ be equal to $(\unlhd_p, b_p, \Delta_p \cup\{N\}, \Omega_p\cup\{N\})$, where $N=H(\kappa^{++}) \cap N^*$. Clearly $q$ is a condition in $\mtcl P_\kappa$ extending $p$.
So, fix an open dense set $D \in N^\ast$ and a condition $q_2=(\unlhd_2,b_2, \Delta_2, \Omega_2)$ extending $q$; we should prove that there exists $q_1$ in $D \cap N^\ast$ compatible with $q_2$. Without loss of generality we will assume that $q_2 \in D$. A second preparation relies on the notion of \emph{height}. For each $(\nu, i) \in \kappa \times \kappa^{++}$, the \emph{height of $(\nu, i)$} is defined as $ht(\nu,i):=i$. If we restrict this function to $dom(\unlhd_2)$, then the order type of the range of this restriction is of course an ordinal below $\kappa$. Let $\pi$ be equal to this order type and let $f_2: \pi \into range(height\restr_{dom(\unlhd_2)})$ be the corresponding enumerating function. So, by elementarity and since $N^\ast$ is closed under small sequences, there is in $N^\ast$ a condition $q_1$ in $D \cap N^\ast$ with the following properties:
\begin{itemize}
\it[(a)] $q_1=(\unlhd_1,b_1, \Delta_1, \Omega_1)$ extends the restriction to $N$ of $q_2$. This restriction is defined as $(\unlhd_2 \cap N, b_2 \cap N, \Delta_2 \cap N, \Omega_2 \cap N)$,\footnote{Note that $N \in \Delta_2 \cap \Omega_2$. So, by the first part of Lemma \ref{iso2} and clause (5) in the definition of the elements of $\mtcl P_\kappa$, this restriction is in fact a condition in $\mtcl P_\kappa$.}
\it[(b)] the order type of the range of $height \restr_{dom(\unlhd_1)}$ is equal to $\pi$,
\it[(c)] there is a function $f_1: \pi \into range(height\restr_{dom(\unlhd_1)})$ enumerating $range(height\restr_{dom(\unlhd_1)})$ in order type $\pi$ and for each $\tau \in \pi$, with $f_2(\tau) \notin N$, $\alpha_{f_2(\tau)}< f_1(\tau)< \beta_{f_2(\tau)}$, where $\alpha_{f_2(\tau)}$ and $\beta_{f_2(\tau)}$ are as in the above lemma for $i=f_2(\tau)$, $N$, and $\mtcl N= \Delta_2$, and
\it[(d)] there is a progressive isomorphism $\Psi_{1,2}: \langle dom(\unlhd_1), \unlhd_1, b_1\rangle \into \langle dom(\unlhd_2), \unlhd_2, b_2 \rangle$.\footnote{With respect to $(\otimes)$ in Definition \ref{progressive} note again that, since $N \in \Omega_2$, $b_2(\{u,v\}) \in N$ whenever $u,v \in N$.}

\end{itemize}

We now claim that $q_1$ and $q_2$ are compatible.  Let $p_3=(\unlhd_3, B_3, \emptyset, \emptyset)$ be the $\Psi_{1,2}$--Amalgamation of $p_1:=(\unlhd_1, b_1, \emptyset, \emptyset)$ and $p_2:=(\unlhd_2, b_2, \emptyset, \emptyset)$ (as in Corollary \ref{amalgcond2}) and define $\Delta_3$ as
$$\Delta_3:= \Delta_2 \cup\{\Psi_{N, N'}(W) \,:\, W \in \Delta_1,\, N' \in \Delta_2,\,\d_{N'}=\d_N\}.$$

So, $p_3 \in \mtcl P_\kappa$ and therefore, $(\unlhd_3, B_3, \Delta_3, \emptyset) \in \mtcl P_\kappa$ (the second assertion follows from $(a)$ above and the second part of Lemma \ref{iso2}). This proof will be completed once we  show that $q_3=(\unlhd_3, B_3, \Delta_3, \Omega_1 \cup \Omega_2)$ satisfies condition $(5)$ of the definition of the elements of $\mtcl P_\kappa$.

So, fix $M \in \Omega_1 \cup \Omega_2$ and $x, y  \in (dom(\unlhd_1) \cup dom(\unlhd_2))\cap M$. We must show that $B_3(\{x,y\})$ is an element of $M$. We omit the obvious case when $M \in \Omega_1$ (note that if $M$ is in $\Omega_1$, then $x$ and $y$ are in $dom(\unlhd_1)$). Assume now that $M \in \Omega_2 \setminus \Omega_1$. If $\delta_M < \delta_N$, then (by (c)) $x$ and $y$ are in $dom(\unlhd_2)$. Since $q_2$ satisfies condition $(5)$ of the definition of the elements of $\mtcl P_\kappa$, $B_3(\{x,y\})=b_2(\{x,y\}) \in M$. The final case to consider is when $\delta_M \geq \delta_N$. In order to avoid triviality, we may assume that $x \in dom(\unlhd_1)$. For each ordinal $\nu \in \kappa^{++}$, let $g_{\nu}$ be the $\Phi$--first surjection between $\kappa^{+}$ and $\nu$. Note that if $x \in dom(\unlhd_1)$, then (by elementarity) all the elements of $dom(\unlhd_1)$ having at most height $\mu:= height(x)$ are of the form $g_{\mu+1}(\varepsilon)$ for some $\varepsilon \in \delta_N$. Since $g_{\mu+1}$ is also definable in $M$ and $\delta_M \geq \delta_N$, the conclusion is that all the elements of $dom(\unlhd_1)$ having at most height $\mu= height(x)$ are in $M$. So,  we have proved the following result.

\begin{remark}
If $x$ and $y$ are in $dom(\unlhd_1)$ with $x \in M$ (independently of whether or not $y \in M$, then $B_3(\{x,y\})=b_1(\{x,y\}) \in M$. More generally, if $x \in M \cap dom(\unlhd_1)$ and $u \unlhd_1 x$, then $u \in M$.
\end{remark}

By the above remark, it remains to verify that if $y \in M \cap dom(\unlhd_2)  \setminus dom(\unlhd_1)$ and $x \in M \cap dom(\unlhd_1) \setminus dom(\unlhd_2)$, then $B_3(\{x,y\}) \in M$ (note that the case when $x \in M \cap dom(\unlhd_2)$ follows from the fact that $M \in \Omega_2$). If $y=\Psi_{1,2}(x)$, then  $B_3(\{x,y\})=B_3(\{x, \Psi_{1,2}(x)\})=\{x\} \in M$. If $y \neq \Psi_{1,2}(x)$, then there are two finite sequences $\langle\{x_i, v_i\} \,:\, i < m \rangle$ and $\langle u_j  \,:\, j < n \rangle$ with  $x_i, u_j \unlhd_1 x$ and $v_i \unlhd_2 y$ for all $i$ and $j$ and such that

$$B_3(\{x, y\})=\bigcup \{b_1(\{x_i, v_i\})\,:\, i <m \} \cup \bigcup \{b_2(\{u_j, y\})\,:\, j < n \}.$$

Using again the above remark we get that, for all $i$ and $j$, $b_1(\{x_i, v_i\}), u_j$ are in $M$. Finally, since $y$ and each $u_j$ are in $M \in \Omega_2$, then so is $b_2(\{u_j, y\})$. This finishes the proof of Theorem \ref{mainthm}.

\end{proof}

\end{document}